\documentclass{elsarticle}

\usepackage[margin=1in]{geometry}

\usepackage{lineno, hyperref, amsmath, amsthm, amssymb, tikz}
\modulolinenumbers[5]

\theoremstyle{definition}
\newtheorem{theorem}{Theorem}[section]

\newtheorem{proposition}{Proposition}[section]

\DeclareMathOperator{\Av}{Av}
\newcommand{\av}[1]{\Av\!\left(#1\right)}
\newcommand{\avh}[1]{\Av_H\!\left(#1\right)}
\newcommand{\avu}[1]{\Av_U\!\left(#1\right)}
\DeclareMathOperator{\Co}{Co}
\newcommand{\co}[1]{\Co\!\left(#1\right)}

\newcommand{\cou}[1]{\Co_U\!\left(#1\right)}
\DeclareMathOperator{\Split}{s}
\renewcommand{\split}[1]{\Split\!\left(#1\right)}

\newcommand{\M}{\mathcal{M}}
\newcommand{\leftp}{\ell}
\newcommand{\rightp}{r}

\DeclareMathOperator{\MinCo}{MinCo}
\newcommand{\minco}[1]{\MinCo\!\left(#1\right)}
\newcommand{\MP}{\mathcal{MP}}

\makeatletter
\def\ps@pprintTitle{%
  \let\@oddhead\@empty
  \let\@evenhead\@empty
  \def\@oddfoot{\reset@font\hfil\thepage\hfil}
  \let\@evenfoot\@oddfoot
}
\makeatother

\begin{document}

\begin{frontmatter}

\title{On the generating functions of pattern-avoiding Motzkin paths\tnoteref{mytitlenote}}
\tnotetext[mytitlenote]{Research partially supported by grant 207178-052 from the Icelandic Research Fund. A.B.~and L.F.~are members of the INdAM research group GNCS. They are partially supported by the 2020-2021 INdAM-GNCS project: ``Combinatoria delle permutazioni, delle parole e dei grafi: algoritmi e applicazioni"}

\author{Christian Bean}
\address{School of Computer Science, Reykjavik University, Reykjavik, Iceland}
\ead{christianbean@ru.is}
\author{Antonio Bernini}
\ead{antonio.bernini@unifi.it}
\author{Luca Ferrari}
\ead{luca.ferrari@unifi.it}
\address{Dipartimento di Matematica e Informatica, Universit\'a di Firenze, Firenze, Italy}
\author{Matteo Cervetti}
\ead{matteocervetti27@gmail.com}
\address{LIB, Universit\'e de Bourgogne Franche-Comt\'e, Dijon, France}




\begin{abstract}

Using a recursive approach, we show that the generating function for sets of Motzkin paths avoiding a single (not necessarily consecutive) pattern 
is rational over $x$ and the Catalan generating function $C(x) = \frac{1-\sqrt{1-4x^2}}{2x^2}$, where $x$ keeps track of the length of the path. Moreover, an algorithm is
provided for finding the generating function in the more general case of an arbitrary set of patterns.
In addition, this algorithm allows us to find a combinatorial specification for pattern-avoiding Motzkin paths,
which can be used not only for enumeration, but also for exhaustive and random generation.
\end{abstract}


\end{frontmatter}


\section{Introduction}
\label{sec:introduction}

A \emph{Motzkin path} of length~$n$ is a lattice path starting at~$(0, 0)$ and
ending at~$(n, 0)$ consisting of \emph{up steps} ($U = (1, 1)$), \emph{down steps} ($D = (1, -1)$) and \emph{horizontal steps} ($H = (1,
0)$) that never goes below the~$x$-axis. We represent Motzkin paths as words over
the alphabet~$\{U, D, H\}$.~In Figure~\ref{fig:motzkinpath}, the Motzkin path
$UHUUDHDD$ is shown.

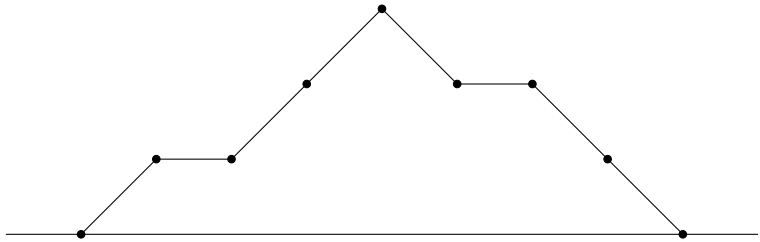
\begin{figure}[h]
    \centering
    \label{fig:motzkinpath}
    \begin{tikzpicture}
        \draw (-1, 0) -- (9, 0);

        \draw (0, 0) -- (1, 1);
        \draw (1, 1) -- (2, 1);
        \draw (2, 1) -- (3, 2);
        \draw (3, 2) -- (4, 3); 
        \draw (4, 3) -- (5, 2); 
        \draw (5, 2) -- (6, 2); 
        \draw (6, 2) -- (7, 1); 
        \draw (7, 1) -- (8, 0); 

        \filldraw (0, 0) circle (0.05);
        \filldraw (1, 1) circle (0.05);
        \filldraw (2, 1) circle (0.05);
        \filldraw (3, 2) circle (0.05);
        \filldraw (4, 3) circle (0.05); 
        \filldraw (5, 2) circle (0.05); 
        \filldraw (6, 2) circle (0.05); 
        \filldraw (7, 1) circle (0.05);
        \filldraw (8, 0) circle (0.05);
    \end{tikzpicture}
    \caption{The Motzkin path $UHUUDHDD$.}
\end{figure}

Let~$\M$ be the set of all Motzkin paths,~$\M_H$ be the set of Motzkin paths
that start with a horizontal step, and~$\M_U$ be the set of Motzkin paths
starting with an up step. With this we have
\begin{equation}\label{eq:motzkinconstruction}
    \M = \{\epsilon\} \sqcup \M_H \sqcup \M_U,
\end{equation}
where $\epsilon$ represents the empty path and we use $\sqcup$ to denote
disjoint union. The ``folklore" result (see for example Donaghey and
Shapiro~\cite{donagheyshapiro:motzkin}) on Motzkin paths says that every Motzkin
path in $\M_H$ can be written as $Hw$ for some $w$ in $\M$, and every Motzkin
path in $\M_U$ can be written $UxDy$ for some $x, y$ in $\M$ as shown in
Figure~\ref{fig:motzkinconstruction}.

\begin{figure}[h]
    \centering
    \label{fig:motzkinconstruction}
    \begin{tikzpicture}
        
        \draw (0, 0) -- (1, 0);
        \draw[dashed] (1, 0) -- node[above] {$\mathcal{M}$} (3, 0);
        \filldraw (0, 0) circle (0.05);
        \filldraw (1, 0) circle (0.05);

        \draw (4, 0) -- (5, 1);
        \filldraw (4, 0) circle (0.05);
        \filldraw (5, 1) circle (0.05);
        \draw[dashed] (5, 1) -- node[above] {$\mathcal{M}$} (7, 1);
        \draw (7, 1) -- (8, 0);
        \filldraw (7, 1) circle (0.05);
        \filldraw (8, 0) circle (0.05);
        \draw[dashed] (8, 0) -- node[above] {$\mathcal{M}$} (10, 0);

        \node (LHS) at (-2, 0.2) {$\mathcal{M}$};
        \node (equal) at (-1.4, 0.2) {$=$};
        \node (epsilon) at (-1.0, 0.2) {$\epsilon$};
        \node (union1) at (-0.5, 0.2) {$\sqcup$};
        \node (union2) at (3.5, 0.2) {$\sqcup$};

    \end{tikzpicture}
    \caption{A pictorial representation of the structural decomposition of Motzkin paths.}
\end{figure}
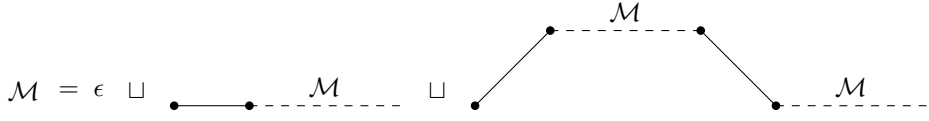

If we let $m_n$ be the number of length $n$ Motzkin paths then it follows that
$m_n$ satisfies the recurrence relation
\begin{equation*}
    m_0 = m_1 = 1,
    \quad
    m_n = m_{n-1} + \sum_{i = 0}^{n - 2} m_{i} m_{n - i - 2}
\end{equation*}
and moreover if we let $M(x) = \sum_{n\geq0} m_n x^n$ be the generating
function, then $M(x)$ satisfies the minimal polynomial
\begin{equation*}
    \label{eq:minpolymotzkin}
    1 + (x - 1) M(x) + x^2 M(x)^2 = 0.
\end{equation*}

Akin to investigations for other combinatorial structures 
(eminently for permutations, but also for graphs), 
there has been interest in studying properties related to the notion of \emph{patterns}
in the context of lattice paths.
For paths it has been common to consider a pattern as a sequence of \emph{contiguous} letters, 
see for instance Asinowski et al.~\cite{asinowskietal:kernel} and 
Sapounakis et al.~\cite{sapounakisetal:string}, 
to cite just a couple of references.
This may be due to the fact that, as we have also remarked above,
the set of lattice paths of a certain type can be conveniently seen as a formal language,
and it is common in theoretical computer science to study contiguous patterns,
or \emph{factors}, of the words of a formal language
(e.g.~\emph{pattern matching} and related problems).
In this work, however,
we will deal with a notion of pattern that is closer to the one usually studied for permutations,
namely we will consider a pattern as a \emph{subword}
(whose letters are not necessarily contiguous) of a given word.
Below we give the necessary notations and definitions in the specific case of Motzkin paths.

A Motzkin path $p$ \emph{contains} a \emph{pattern} $q$ in $\{U, H, D\}^*$,
written $q \preceq p$, if $q$ occurs as a subword in $p$.~If $p$ does not
contain $q$ we say $p$ \emph{avoids} $q$ and write $q \npreceq p$.~For a set $P$
of patterns, we say a path \emph{avoids} $P$ if it avoids all $q \in P$ and
define the set of Motzkin paths avoiding $P$ as
\begin{equation*}
    \av{P} = \{p \in \M ~|~ p ~ \text{avoids} ~ P\}.
\end{equation*}
If a path does not avoid $P$ we say it \emph{contains} $P$ and define the set of
Motzkin paths containing $P$ as
\begin{equation*}
    \co{P} = \{p \in \M ~|~ p ~ \text{contains} ~ P\},
\end{equation*}
i.e.~the set of Motzkin paths containing at least one pattern in $P$.~For instance, 
consider again the Motzkin path in Figure~\ref{fig:motzkinpath}. 
It is easy to check that it contains the pattern $UHDH$, 
but it avoids the pattern $UDHH$.~Notice that, according to our definition of containment for sets of patterns, such a path also contains the set $P=\{ UHDH,UDHH\}$.

The set $\av{H}$\footnote{We will not include the braces in our notation,
i.e.~we write $\av{H}$ rather than $\av{\{H\}}$} is the set of Dyck paths which
are counted by the Catalan numbers.
Recall that Dyck paths are defined like Motzkin paths, except that they do not use horizontal steps.
In Bacher et al.~\cite{bacheretal:dyck}, it was shown that any set of
Dyck paths avoiding a single pattern has a rational generating function. In
this paper, we show a similar statement holds for the set of Motzkin paths,
alongside an algorithm for effectively computing the generating functions even in the more general case of a set of patterns.

\begin{theorem}
    \label{thm:motzkinarealmostrational}
    Let $q$ be a fixed pattern and let $a_n$ be the number of $q$-avoiding Motzkin paths of length $n$.
    Then the generating function $\Delta_q (x)= \sum_{n\geq 0} a_n
    x^n$ is rational over $x$ and $C(x) = \sum_{n\geq 0} C_n
    x^{2n}$, where $C_n$ is the $n$-th Catalan number.
\end{theorem}

The paper is organised as follows. In Section~\ref{sec:algorithm}, we
outline an algorithm for computing a \emph{combinatorial specification}, in the sense of Flajolet and Sedgewick~\cite{flajoletsedgewick:ac}, for
sets of Motzkin paths avoiding an arbitrary set of patterns. Such a
specification then gives a method for computing the generating function but also the
ability to sample uniformly from these sets. In Section~\ref{sec:proof}, we
give a proof of Theorem~\ref{thm:motzkinarealmostrational}. 
The strategy of the proof consists of describing a recursive procedure to compute the generating function $\Delta_q (x)$;
such a procedure also depends on certain auxiliary generating functions,
which are in turn described in a recursive fashion. 
Finally, in Section \ref{sec:conclusion} some suggestions for further research are given.

In closing this Introduction, we remark that, whenever we will consider Dyck paths, 
we will usually enumerate them according to the length, 
rather than (as it is usual) the semilength.
As a consequence, our version of the Catalan generating function is 
$C(x) = \frac{1-\sqrt{1-4x^2}}{2x^2}$, 
hence, for all $n\in \mathbb{N}$, $[x^{2n}]C(x)=C_{n} = \frac{1}{n+1} \binom{2n}{n}$ (the $n$-th Catalan number)
and $[x^{2n+1}]C(x)=0$ (where $[x^n]F(x)$ denotes the $n$-the coefficient of the generating function $F(x)$).
%
%
%

\section{The algorithm}
\label{sec:algorithm}

For a set $P$ of patterns, we define $\avh{P}$ to be the Motzkin paths avoiding $P$ and beginning with an $H$ step,
\begin{equation*}
    \avh{P} = \av{P} \cap \M_H,
\end{equation*}
and $\avu{P}$ to be the Motzkin paths avoiding $P$ and beginning with a $U$ step,
\begin{equation*}
    \avu{P} = \av{P} \cap \M_U.
\end{equation*}
The set of Motzkin paths avoiding $P$ can be partitioned in the same manner as the set of
all Motzkin paths,
\begin{equation}\label{eq:avodingdecomposition}
    \av{P} = \{\epsilon\} \sqcup \avh{P} \sqcup \avu{P}.
\end{equation}

\subsection{Starting with H}

Using the following, we can enumerate the set $\avh{P}$.

\begin{theorem}
\label{thm:startwithH}
For a set of patterns $P$, let $P_U$, $P_D$, and $P_H$ be the sets of patterns in $P$ beginning with $U$, $D$ and $H$, respectively, and
\begin{equation*}
    P' = P_U \cup P_D \cup \{p ~ | ~ Hp \in P_H \}
\end{equation*}
then
\begin{equation*}
    \avh{P} = \{ Hp ~ | ~ p \in \av{P'} \}.
\end{equation*}
\end{theorem}

\begin{proof}
    Let $Hq$ be a path in $\avh{P}$.~It follows that $q$ must avoid every pattern in $P_U$ and $P_D$.~If $q$ contains an occurrence of some $p$ for $Hp \in P_H$, then $Hq$ must contain an occurrence of $Hp$.~Therefore, $q$ must avoid $p$. 

    On the contrary, let $Hq \in \{ Hp ~ | ~ p \in \av{P'} \}$, then as $q$ avoids $P_U$ and $P_D$, it follows that $Hq$ also avoids $P_U$ and $P_D$.~As $q$ avoids all $p$ coming from $Hp \in P_H$, the path $Hq$ must avoid $Hp$.~Hence we have shown that $Hq$ avoids $P$.  
\end{proof}

If $a_n$ counts the number of length $n$ paths in $\avh{P}$ and $b_n$ counts the
number of length $n$ paths in $\av{P'}$ then by the above Theorem we get $a_k = b_{k-1}$
for all $k \geq 1$.

\subsection{Starting with U}

Every Motzkin path in $\avu{P}$ can be written as $UxDy$ for some $x, y$ in
$\av{P}$, however not each choice of $x$ and $y$ from $\av{P}$ yields a Motzkin path in $\avu{P}$.~For example, if $P
= \{HH\}$, and $x = H$ and $y = H$, then we get the path $UHDH$ which contains an
occurrence of $HH$.~In order to capture this, we introduce the notion of
\emph{crossing patterns}.

A Motzkin path $p \in \M_U$ \emph{contains} the \emph{crossing pattern} $\leftp
- \rightp$, where $\leftp$ and $\rightp$ are words over the alphabet $\{U, D,
H\}$, if $p$ can be written $UxDy$ where $UxD$ contains $\leftp$ and $y$
contains $\rightp$.~Otherwise, we say it \emph{avoids} $\leftp - \rightp$.~If
either $\leftp$ or $\rightp$ is $\epsilon$ we write $-\rightp$ and $\leftp-$,
respectively, and call these patterns \emph{local}. We use the notation of
$\av{P}$ and $\co{P}$ as before for crossing patterns.

With our new definition, for the case of $P = \{ HH\}$ we have
\begin{equation*}
    \avu{P} = \avu{-HH, H-H, HH-}.
\end{equation*} 

For a path $UxDy$ in $\avu{P}$, if $x$ avoids $H$ then $UxDy$ avoids $H-H$ and $HH-$.
However, if $x$ contains $H$ then
$UxDy$ avoids $-H$ and $HH-$. 
That is
\begin{equation}
    \label{eq:avoidcrossingHH}
    \avu{-HH, H-H, HH-} =
    \avu{-HH, H-}
    \sqcup
    \left( \avu{-H, HH-} \cap \co{H-} \right).
\end{equation}
Every path in $\avu{-HH, H-}$ can be written $UxDy$ where $x \in \av{H}$ and $y
\in \av{HH}$.~Similarly, every path in $\avu{-H, HH-} \cap \co{H-}$ can be
written $UxDy$ where $x \in \av{HH} \cap \co{H}$ and $y \in \av{H}$.
This argument is shown pictorially in Figure~\ref{fig:avoidcrossingHH}.

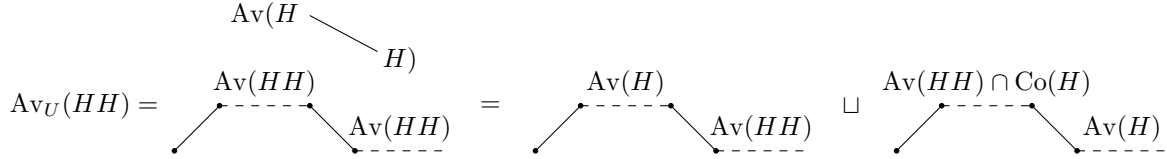
\begin{figure}[h]
   
    \begin{tikzpicture}[scale=0.6]
        \node (LHS) at (0, 1) {$\avu{HH} =$};
        \draw (2, 0) -- (3, 1);
        \filldraw (2, 0) circle (0.05);
        \filldraw (3, 1) circle (0.05);
        \draw[dashed] (3, 1) -- node[above] {$\av{HH}$} (5, 1);
        \draw (5, 1) -- (6, 0);
        \filldraw (5, 1) circle (0.05);
        \filldraw (6, 0) circle (0.05);
        \draw[dashed] (6, 0) -- node[above] {$\av{HH}$} (8, 0);

        \node (start) at (4, 3) {$\Av(H$};
        \node (end) at (7, 2) {$H)$};
        \draw (5,3) -- (6.5, 2.2); 

        \node (equals) at (9, 1) {$=$};

        \draw (10, 0) -- (11, 1);
        \filldraw (10, 0) circle (0.05);
        \filldraw (11, 1) circle (0.05);
        \draw[dashed] (11, 1) -- node[above] {$\av{H}$} (13, 1);
        \draw (13, 1) -- (14, 0);
        \filldraw (13, 1) circle (0.05);
        \filldraw (14, 0) circle (0.05);
        \draw[dashed] (14, 0) -- node[above] {$\av{HH}$} (16, 0);

        \node (union) at (17, 1) {$\sqcup$};

        \draw (18, 0) -- (19, 1);
        \filldraw (18, 0) circle (0.05);
        \filldraw (19, 1) circle (0.05);
        \draw[dashed] (19, 1) -- node[above] {$\av{HH} \cap \co{H}$} (21, 1);
        \draw (21, 1) -- (22, 0);
        \filldraw (21, 1) circle (0.05);
        \filldraw (22, 0) circle (0.05);
        \draw[dashed] (22, 0) -- node[above] {$\av{H}$} (24, 0);
    \end{tikzpicture}
    \caption{A pictorial representation of Equation~(\ref{eq:avoidcrossingHH}) for $\avu{HH}$.} \label{fig:avoidcrossingHH}
\end{figure}

Theorem~\ref{thm:startwithH} says that 
\begin{equation}\label{eq:startingwithHavoidingHH}
    \avh{HH} = \{Hp | p \in \av{H}\}.    
\end{equation}

Let $\Delta_{HH}(x)$
be the generating function for $\av{HH}$ and $C(x)=\frac{1-\sqrt{1-4x^2}}{2x^2}$ be the generating function
for $\av{H}$,  
then it follows from Equations~\eqref{eq:avodingdecomposition}, \eqref{eq:avoidcrossingHH}, and \eqref{eq:startingwithHavoidingHH} that $\Delta_{HH}(x)$ satisfies the equation
\begin{equation}\label{eq:exampleHH}
    \Delta_{HH}(x) = 1 + xC(x) + x^2C(x)\Delta_{HH}(x) + x^2(\Delta_{HH}(x) - C(x))C(x).
\end{equation}
From~\eqref{eq:exampleHH}, replacing $C(x)$ with $\frac{1-\sqrt{1-4x^2}}{2x^2}$ and squaring after suitable manipulations, we find that $\Delta_{HH}(x)$ satisfies the minimal polynomial
\begin{equation*}
    (4x^4 - x^2)\Delta_{HH}(x)^2 + (4x^3 - 4x^2 - x + 1)\Delta_{HH}(x) + 5x^2 - 1.
\end{equation*}

In passing, we observe that the generating function $\Delta_{HH}(x)$ is interesting in itself. In fact, the coefficients of even index are Catalan numbers, whereas the coefficients of odd index (which count, by the way, Motzkin paths having exactly one horizontal step with respect to the length) are the binomial coefficients $\binom{2n+1}{n+1}$, i.e.,~sequence A001700 in the Online Encyclopedia of Integer Sequences (OEIS) \cite{sloane}.

We generalize this idea in the following theorem.

\begin{theorem}
    \label{thm:avoidorcontain}
    For any finite sets $P$ and $Q$ of patterns there exist sets of local
    crossing patterns $P_1$, $P_2$, $\ldots$, $P_k$ and $Q_1$, $Q_2$, $\ldots$,
    $Q_k$ such that
    \begin{equation*}
        \avu{P} \cap \bigcap_{q \in Q} \co{q} = \bigsqcup_{i = 1}^k
        \left(
            \avu{P_i} \cap \bigcap_{q \in Q_i} \cou{q}
        \right).
    \end{equation*}
\end{theorem}

\begin{proof}
    For a pattern $p$ we define the set $\split{p}$ consisting of all crossing
    patterns $\leftp-\rightp$ such that $\leftp\rightp$ is $p$,
    \begin{equation*}
        \split{p} = \{
            \leftp-\rightp ~ | ~ \leftp, \rightp \in \{U, D, H\}^*
                                 \text{ and } \leftp\rightp = p
        \}.
    \end{equation*}
    It follows that every Motzkin path in $\avu{P}$ must avoid
    $\bigsqcup_{p \in P} \split{p}$ and every Motzkin path in $\bigcap_{q \in Q} \co{q}$ must contain some pattern in $\split{q}$ for
    each $q$ in $Q$, i.e.~
    \begin{equation*}
        \avu{P} \cap \bigcap_{q \in Q} \co{q} = \avu{\bigsqcup_{p \in P} \split{p}} \cap
        \bigcap_{q \in Q} \cou{\split{q}}.
    \end{equation*}
    
    For a set of patterns $Q$ we can partition the set of Motzkin paths
    containing $Q$ into those that avoid a pattern $q$ in $Q$ and those that
    contain $q$.~In our case, if we set $Q=\{ q_1 ,q_2 ,\ldots ,q_h \}$, we are avoiding a set of crossing patterns $\widetilde{P}=\bigsqcup_{p \in P} \split{p}$ and must
    contain the sets of crossing patterns $\widetilde{Q}_1$, $\widetilde{Q}_2$, $\ldots$, $\widetilde{Q}_h$, where $\widetilde{Q}_i =\split{q_i}$, for all $i=1,2,\ldots ,h$.~Therefore, if $q \in
    \widetilde{Q}_1$, we have
    \begin{align*}
        \avu{\widetilde{P}} \cap \bigcap_{i=1}^h \cou{\widetilde{Q}_i}
        =&
        \left(
            \avu{\widetilde{P} \cup \{q\}} \cap \cou{\widetilde{Q}_1 \backslash \{q\}} \cap
            \bigcap_{i=2}^h \cou{\widetilde{Q}_i}
        \right) \\
        & \sqcup
        \left(
            \avu{\widetilde{P}} \cap \cou{\{q\}} \cap \bigcap_{i=2}^h \cou{\widetilde{Q}_i}
        \right).
    \end{align*}
    By iterating this process, and perhaps rearranging the order of the $\widetilde{Q}_i$,
    this will result in writing 
    \begin{equation*}
        \avu{P} \cap \bigcap_{q \in Q} \co{q} = \bigsqcup_{i = 1}^h
        \left(
            \avu{\widetilde{P}_i} \cap \bigcap_{q \in \widetilde{Q}_i} \cou{q}
        \right).
    \end{equation*}
    where the $\widetilde{P}_i$ and $\widetilde{Q}_i$ are sets of (not necessarily local) crossing
    patterns. 

    If $\leftp-\rightp$ in $\widetilde{Q}_1$ then it follows that the paths must contain both $\leftp-$ and $-\rightp$, i.e.~they contain every pattern in the set $\widetilde{Q}_1' = \left( \widetilde{Q}_1 \backslash
    \{\leftp-\rightp\}\right) \sqcup \{\leftp-, -\rightp\}$ and we have
    \begin{equation}
        \label{eq:crossingcontain}
        \avu{\widetilde{P}_1} \cap \bigcap_{q \in \widetilde{Q}_1} \cou{q} = 
        \avu{\widetilde{P}_1} \cap \bigcap_{q \in \widetilde{Q}_1'} \cou{q}.
    \end{equation}
    If $\leftp-\rightp$ in $\widetilde{P}_1$ then we can partition the paths to those that avoid $\leftp-$ and those that contain $\leftp-$ giving
    \begin{align}
        \label{eq:crossingavoid}
        \begin{split}
        \avu{\widetilde{P}_1} \cap \bigcap_{q \in \widetilde{Q}_1} \cou{q} =& 
        \left(
            \avu{\widetilde{P}_1 \sqcup \{\leftp-\}} \cap \bigcap_{q \in \widetilde{Q}_1} \cou{q}
        \right) \\ &\sqcup 
        \left(
            \avu{\widetilde{P}_1 \sqcup \{-\rightp\}} \cap \bigcap_{q \in \widetilde{Q}_1 \sqcup \{\leftp-\}} \cou{q}
        \right).
        \end{split}
    \end{align}
    By repeated application of Equations~(\ref{eq:crossingcontain}) and
    (\ref{eq:crossingavoid}), simplifying the avoidance and containment sets, and possibly reordering the $\widetilde{P}_i$ and $\widetilde{Q}_i$, we get
    the desired disjoint union where all of the crossing patterns in the resulting sets $P_i$ and $Q_i$ are local. 
\end{proof}

The case analysis of Theorem~\ref{thm:avoidorcontain} will result in a disjoint union
with sets of the form $\avu{P} \cap \bigcap_{q \in Q} \cou{q}$ for some sets of
local crossing patterns $P$ and $Q$.~As the pattern containment conditions for these sets of Motzkin paths become local we get the following theorem.

\begin{theorem}
    \label{thm:cartesianproduct}
    Let $P$ and $Q$ be sets of local crossing patterns. Let $P_r$ ($Q_r$) be the
    right local patterns in $P$ ($Q$). Let $P_\ell$ ($Q_\ell$) be the patterns
    obtained by taking the left local patterns in $P$ ($Q$) and removing a
    single $U$ from the left and single $D$ from the right if such exists. Then,
    \begin{equation*}
        \avu{P} \cap \bigcap_{q \in Q} \cou{q} =
        \{
            UxDy ~ | ~ 
            x \in \avu{P_\ell} \cap \bigcap_{q \in Q_\ell} \cou{q}, 
            y \in \avu{P_r} \cap \bigcap_{q \in Q_r} \cou{q}
        \}.
    \end{equation*}
\end{theorem}

Let $a_n$ be the number of length $n$ Motzkin paths in $\av{P} \cap \bigcap_{q
\in Q} \co{q}$, $b_n$ be the number of length $n$ Motzkin paths in $\av{P_\ell}
\cap \bigcap_{q \in Q_\ell} \co{q}$, and $c_n$ be the number of length $n$
Motzkin paths in $\av{P_r} \cap \bigcap_{q \in Q_r} \cou{q}$.
Theorem~\ref{thm:cartesianproduct} implies $a_n = \sum_{i = 0}^{n-2}
b_ic_{n-i-2}$.

\subsection{Finding a specification}

\emph{Combinatorial exploration}, introduced in Bean~\cite{bean:thesis}, is an
automatic method for finding \emph{(combinatorial) specifications}. 
It consists of a systematic
application of \emph{strategies} to create \emph{(combinatorial) rules} about a
\emph{(combinatorial) class} of interest. Each rule describes how to build a
class from other classes using well-defined \emph{constructors}. In this paper,
we only use the disjoint union and Cartesian product constructors. Using these
rules, the method then finds a specification which can be used, for example, to
count the number of objects of each size, generate objects, and sample uniformly
at random. This entire procedure has been implemented as the
\verb|comb_spec_searcher| Python package by Bean et al.~\cite{bean:cssrepo}. 

Theorems~\ref{thm:startwithH}, \ref{thm:avoidorcontain}, and
\ref{thm:cartesianproduct} encode strategies for finding specifications for
pattern-avoiding Motzkin paths. Moreover, as the recursive application of these
theorems result in either shortening of the patterns being avoided and contained
or reducing the size of the sets being avoided and contained, they give a finite
process that will always result in a specification. Our Python implementation, which
uses the \verb|comb_spec_searcher| package, can be found on GitHub by Bean~\cite{bean:motzkinrepo}. 

We do one final example to illustrate the recursive nature of the theorems.

\subsection{Enumerating Motzking paths avoiding UHHD}

In this section, we outline how Theorems~\ref{thm:startwithH}, \ref{thm:avoidorcontain},
and \ref{thm:cartesianproduct} are used to enumerate $\av{UHHD}$. 

We first apply Equation~\eqref{eq:avodingdecomposition} to get 
\begin{equation}
    \label{eq:UHHDpartition}
    \av{UHHD} = \{\epsilon\} \sqcup \avh{UHHD} \sqcup \avu{UHHD}.
\end{equation}
Theorem~\ref{thm:startwithH} tells us that 
\begin{equation}
    \label{eq:UHHDstartH}
    \avh{UHHD} = \{ H \} \times \av{UHHD}.
\end{equation}
Theorem~\ref{thm:avoidorcontain} tells us that we can find sets of local crossing patterns to describe $\avu{UHHD}$.~We follow the algorithm outlined in the proof of the theorem to get this description. The set of crossing patterns coming from $UHHD$ is 
\begin{equation*}
    s(UHHD) = \{-UHHD, U-HHD, UH-HD, UHH-D, UHHD- \}.
\end{equation*}
Therefore, $\avu{UHHD} = \avu{s(UHHD)}$.~As we have no sets of patterns $Q$ to contain as in Theorem~\ref{thm:avoidorcontain}, we apply Equation~\ref{eq:crossingavoid} to get 
\begin{align}
    \label{eq:UHHDzerothiteration}
    \begin{split}
            \avu{s(UHHD)} &= \avu{-UHHD, -HHD, UH-} \cap \cou{U-} \\
    & \phantom{=} \sqcup 
    \left( 
        \avu{-HD, UHH-D, UHHD-}
        \cap 
        \cou{UH-}
    \right) \\
    &= \avu{-HHD, UH-} \cap \cou{U-} \\
    & \phantom{=} \sqcup 
        \avu{-HD, UHH-}
        \cap 
        \cou{UH-}
    \\
    & \phantom{=} \sqcup 
    \avu{-D, UHHD-}
    \cap 
    \cou{UHH-}
    \end{split}
\end{align} 
Note, the case where you are avoiding $U-$ is precisely the empty set since all of the paths contain $U-$, and therefore we have not included this in our equations.

Each of the disjoint sets on the right are defined by local crossing patterns and so we apply Theorem~\ref{thm:cartesianproduct} to each of these sets. This gives the equations
\begin{align}
    \label{eq:UHHDfirstiteration}
    \begin{split}
    \avu{-HHD, UH-} \cap \cou{U-} 
    & = 
    \{UD\} 
    \times 
    \left( 
        \av{H} \cap \co{\epsilon} 
    \right) 
    \times
    \av{HHD}  \\
    \avu{-HD, UHH-}
        \cap 
    \cou{UH-}
    & = 
    \{UD\}
    \times 
    \left(
        \av{HH} \cap \co{H}
    \right) 
    \times 
    \av{HD} \\
    \avu{-D, UHHD-}
    \cap 
    \cou{UHH-} 
    & = 
    \{UD\} 
    \times 
    \left(
        \av{HH} \cap \co{HH}
    \right) 
    \times 
    \av{D}.
    \end{split}
\end{align}
As a path can not both avoid and contain $HH$, the latter set is in fact empty. The algorithm would actually spot this sooner as it is not hard to argue that a path cannot simultaneously avoid $UHHD-$ and contain $UHH-$ since the final letter before the split in a path is $D$. 

In the previous section we enumerated $\av{HH} \cap \co{H}$ by utilising the set difference 
\begin{equation*}
    \av{HH} \cap \co{H} = \av{HH} \backslash \av{H}
\end{equation*}
which implies its generating function is $\Delta_{HH}(x) - C(x)$.~Although this is a legitimate method for enumeration, due to the set difference operation above, the sampling of the Motzkin paths will not be efficient.
Therefore our algorithm instead continues to apply the theorems in order to only use disjoint unions and Cartesian products. 

In Equations~\eqref{eq:UHHDfirstiteration}, we need to expand further the sets $\av{HH} \cap \co{H}$, $\av{HHD}$, and $\av{HD}$.~We first apply Theorem~\ref{thm:startwithH} to get 
\begin{align}
    \label{eq:UHHDsecondstartH}
    \begin{split}
    \avh{HH} \cap \co{H} & = \{H\} \times \av{H} \\
    \avh{HHD} & = \{H\} \times \av{HD}\\
    \avh{HD} & = \{H\} \times \av{D}.
    \end{split}
\end{align}
Note, $\av{D}$ is precisely $\{H\}^*$.

We then apply Theorem~\ref{thm:avoidorcontain} to get 
\begin{align}
    \label{eq:UHHDseconditeration}
    \begin{split}
    \avu{HH} \cap \co{H} & = \avu{-HH, H-H, HH-} \cap \cou{-H, H-} \\
    & = \left(
            \avu{H-, -HH} \cap \cou{-H} 
        \right)
        \sqcup 
        \left( 
            \avu{-H, HH-} \cap \cou{H-}
        \right) \\
    \avu{HD} & = \avu{-HD, H-D, HD-} \\
    & = \avu{H-, -HD} \sqcup 
    \left( 
        \avu{-D, HD-} \cap \cou{H-} 
    \right)\\
    \avu{HHD} & = \avu{-HHD, H-HD, HH-D, HHD-} \\
    & = \avu{H-, -HHD} \sqcup 
    \left(
        \avu{-HD, HH-D, HHD-} \cap \cou{H-}
    \right) \\
    & = \avu{H-, -HHD} \sqcup 
        \left(
            \avu{-HD, HH-} \cap \cou{H-}
        \right) \\
        & \phantom{=}
        \sqcup 
        \left(
            \avu{-D, HHD-} \cap \cou{HH-}
        \right).
    \end{split}
\end{align}
By a similar argument as before the sets $\avu{-D, HD-} \cap \cou{H-}$ and $\avu{-D, HHD-} \cap \cou{HH-}$ are the empty set. 
For the remaining all of the patterns are local crossing patterns so we can apply Theorem~\ref{thm:cartesianproduct}.
\begin{align}
    \label{eq:UHHDfinaliteration}
    \begin{split}
    \avu{H-, -HH} \cap \cou{-H} 
    & = \{UD\} \times \av{H} \times \left( \av{HH} \cap \co{H} \right) \\
    \avu{-H, HH-} \cap \cou{H-} 
    & = \{UD\} \times \left( \av{HH} \cap \co{H} \right) \times \av{H} \\
    \avu{H-, -HD} 
    & = \{UD\} \times \av{H} \times \av{HD} \\
    \avu{H-, -HHD}
    & = \{UD\} \times \av{H} \times \av{HHD} \\
    \avu{-HD, HH-} \cap \cou{H-} 
    & = \{UD\} \times \left( \av{HH} \cap \co{H} \right) \times \av{HD}
    \end{split}
\end{align}
The Equations~\eqref{eq:UHHDpartition}, \eqref{eq:UHHDstartH}, \eqref{eq:UHHDzerothiteration}, \eqref{eq:UHHDfirstiteration}, \eqref{eq:UHHDsecondstartH}, \eqref{eq:UHHDseconditeration}, and \eqref{eq:UHHDfinaliteration} give a combinatorial specification for $\av{UHHD}$, and can be used directly to get the generating function 
\begin{equation*}
    \frac{1 - 3x - 4x^2 + 12x^3 - (1 - 3x - 4x^2 + 8x^3)\sqrt{1 - 4x^2}}{2x^2 (1 - 2x -3x^2 + 8x^3 - 4x^4)}
\end{equation*}
for this set. The coefficients of this generating function are the sequence A347036 in the OEIS~\cite{sloane}. We ran our algorithm on many sets of pattern-avoiding Motzkin paths. We list a few which have connections to sequences in the OEIS~\cite{sloane}. 

The generating function for $\av{UDH}$ is 
\begin{equation*}
    \frac{1 - 2x - \sqrt{1-4x^2}}{2x(2x - 1)}
\end{equation*}
which shows that there are $\binom{n}{\lfloor \frac{n}{2} \rfloor}$ paths of length $n$ in this set, i.e.~sequence A001405 in the OEIS~\cite{sloane}. 

The set $\av{UDH, UHD}$ has the generating function 
\begin{equation*}
    \frac{1 - \sqrt{1 - 4x^2}}{2x^2(x - 1)}
\end{equation*}
whose coefficients are the sequence A110199 in the OEIS~\cite{sloane}. This tells us that there are $\sum_{k = 0}^{n} C_n$ many paths of lengths $2n$ and $2n + 1$ in $\av{UDH, UHD}$. This follows from the fact that the paths in this set can be described as some Dyck path prepended with an arbitrary number of $H$ steps. 

The generating functions for $\av{UUDD}$ and $\av{UDUDUD, UUDDUD, UUDUDD}$ are 
\begin{equation*}
    \frac{1 - 4x + 7x^2 - 6x^3 + 3x^4}{(1-x)^5} \text{ and } \frac{1 - 4x + 6x^2 - 4x^3 + 2x^4}{(1-x)^3(1-2x)}
\end{equation*}
whose coefficients are the sequences A000127 and A084634 in the OEIS~\cite{sloane}. 

In a few cases, we found the sequences for the coefficients at odd indices in the generating functions appeared in the OEIS~\cite{sloane}. For example, if we let $a_n$ be the number of length $n$ Motzkin paths in $\av{HHUD, HUHD, UHHD}$
then the generating function for the sequence of odd length paths in this set, i.e.,~$\sum_{n \geq 0} a_{2n + 1} x^n$, is
\begin{equation*}
    \frac{1 - 4x - \sqrt{1-4x}}{-2x(1 - 5x + 4x^2)}.
\end{equation*}
The coefficients are the sequence A079309 in the OEIS~\cite{sloane}.
Similarly, the sequence given by the number of odd length paths in $\av{UDHH}$ has the generating function
\begin{equation*} 
    \frac{1 - 6x + 8x^2 - (1 - 3x)\sqrt{1 - 4x}}{- x + 8x^2 - 16x^3}
\end{equation*}
which appears to be the sequence A194460 in the OEIS~\cite{sloane}.

\section{Proof of Theorem \ref{thm:motzkinarealmostrational}}
\label{sec:proof}


In this section, we give a proof of Theorem \ref{thm:motzkinarealmostrational}.
Analogously to what has been done in Bacher et al.~\cite{bacheretal:dyck} for Dyck paths, 
our strategy is the following.
First, we describe functional equations satisfied by some bivariate generating functions of certain Motzkin prefixes, 
where the relevant statistics are the length and the final height. 
These are then used to find functional equations for the generating functions of Motzkin paths avoiding a single pattern.
The derived equations will clearly show that such generating functions are rational over $x$ and $C(x)$, as desired.

\bigskip

A \emph{Motzkin prefix} is defined exactly like a Motzkin path, except that the final point of the path has nonnegative height (so it is not required that the path ends on the $x$-axis, but it can end at every point having nonnegative integer coordinates).
Denote with $\MP$ the set of all Motzkin prefixes.
For a given non-empty Motzkin prefix $p$, let $p^-$ be the Motzkin prefix obtained from $p$ by removing its last step. 
Given a set $P$ of Motzin prefixes, let $\minco{P}=\{ p\in \MP ~|~ p ~ \text{contains} ~ P ~ \text{and} ~ p^- ~ \text{avoids} ~ P  \}$.~In other words, an element of $\minco{P}$ is a smallest Motzkin prefix containing $P$.~
In the sequel, we will be interested in the case where $P=\{ q\}$, for a certain Motzkin prefix $q$.

Denote with $\Gamma_q (x,y)$ the bivariate generating function of the smallest Motzkin prefixes containing $q$,
where $x$ keeps track of the length and $y$ keeps track of the final height.
For instance, choosing $q=UH$, the generic smallest Motzkin prefix containing $q$ is obtained by concatenating
a sequence of letters $H$ with a non-empty Dyck prefix followed by an $H$.~Thus, recalling the expression of the bivariate generating function
$\mathcal{DP}(x,y)=\frac{2}{1-2xy+\sqrt{1-4x^2}}$ of Dyck prefixes, we get

$$\Gamma_q (x,y)=\frac{1}{1-x}\left( \frac{2}{1-2xy+\sqrt{1-4x^2}}-1\right)x .$$

The following result gives a recursive procedure to compute $\Gamma_q (x,y)$.
In the statement below, $\epsilon$ is the empty path; 
moreover, given $q\in \MP$ and $X\in \{ U,H,D\}$, $qX$ is the Motzkin prefix obtained by appending the step $X$ to $q$.

\begin{proposition}\label{Prop_C_P}
	For any given Motzkin prefix $q$, we have:
	\begin{align}
	\Gamma_{\epsilon}(x,y)&=1, \\
	\label{C_qU}
	\Gamma_{qU}(x,y)&=\frac{xy}{(1-x)(x-y(1-x))}\left( x\Gamma_q \left( x,\frac{x}{1-x}\right) -y(1-x)\Gamma_q (x,y) \right) , \\
	\label{C_qH}
	\Gamma_{qH}(x,y)&=\frac{2x}{\left(1-2xy+\sqrt{1-4x^2}\right)\left(y-xC(x)\right)}\left( y\Gamma_q (x,y)-xC(x)\Gamma_q (x,xC(x)) \right) , \\
	\label{C_qD}
	\Gamma_{qD}(x,y)&=\frac{x}{y}\left( \frac{1}{1-xy-x}\Gamma_q (x,y)-\frac{1}{1-x}\Gamma_q (x,0) \right) .
	\end{align}
\end{proposition}

\begin{proof}
Clearly, the only smallest Motzkin prefix containing the empty path is $\epsilon$ itself, which gives $\Gamma_{\epsilon}(x,y)=1$.

\medskip
Let $\pi \in \minco{qU}$ and denote with $\pi'$ the smallest prefix of $\pi$ containing $q$.~Moreover, we indicate with $h$ the final height of $\pi'$.~Then, $\pi$ can be factorized as
\begin{eqnarray}\label{Q_CPU}
\pi =\pi' \beta^{(h)}U,
\end{eqnarray}
where $\beta^{(h)}$ is a path starting at height $h$ (which is the height of the final point of $\pi'$) using only $H$ steps and $D$ steps and not crossing the $x$-axis (i.e., the final height $i$ of $\beta^{(h)}$ is such that $0\leq i \leq h$). Clearly, the path $\beta^{(h)}$ is the reverse of a path $\alpha$ starting at the origin, using only $U$ steps and $H$ steps, with final height less than or equal to $h$.~It is not difficult to compute the bivariate generating function $A(x,y)$ of such paths $\alpha$, where $x$ and $y$ track the length and the final height of $\alpha$, respectively. Indeed, such a path $\alpha$, if not empty, can be obtained either by taking an $H$ step followed by a pattern of the same kind or by taking a $U$ step followed by a pattern of the same kind. This leads to the functional equation:
$$
A(x,y)=1+xA(x,y)+xyA(x,y),\
$$
hence
\begin{eqnarray}\label{A(x,y)}
A(x,y)=\frac{1}{1-x-xy}=\frac{1}{1-x}\sum_{n\geq 0}\left(\frac{x}{1-x}\right)^ny^n\ ,
\end{eqnarray}

If $B^{(h)}(x,y)$ denotes the generating function of the paths $\beta^{(h)}$ (where $x$ and $y$ have the same role as in $A(x,y)$), using essentially the same argument as above, we have
$$
B^{(h)}(x,y)=\sum_{i=0}^h \left( \left[y^{h-i}\right] A(x,y)\right) y^i ;
$$
hence, in terms of generating functions, relation (\ref{Q_CPU}) becomes:

	\begin{align}\label{CPU}
	\Gamma_{qU}(x,y)&=
	\left(
	\sum_{h\geq0}\left( \left[ y^h \right] \Gamma_q (x,y)\right) B^{(h)}(x,y)\right)xy \\
	&=\left(
	\sum_{h\geq0}\left( \left[y^h\right]\Gamma_q (x,y)\right) \sum_{i=0}^h\left( \left[y^{h-i}\right]A(x,y)\right) y^i\right)xy\ \nonumber.
	\end{align}

\noindent
We note that, referring to (\ref{Q_CPU}), the term $\displaystyle \sum_{h\geq0}\left[y^h\right]\Gamma_q (x,y)$ in (\ref{CPU}) records the prefix $\pi'$, while the term $xy$ tracks the step $U$.
By using (\ref{A(x,y)}) for the coefficient $\left[y^{h-i}\right]A(x,y)$, expression (\ref{CPU}) can be reduced to (\ref{C_qU}).

\bigskip

Similarly, let $\pi \in \minco{qH}$ and let $\pi '$ be the smallest prefix of $\pi$ containing $q$.~We denote with $\delta ^{(h)}$ a Dyck factor starting at height $h$ (i.e., a sequence of $U$ and $D$ steps which does not cross the $x$-axis) and with $D ^{(h)}(x,y)$ the bivariate generating function for such paths. We have:
$$
\pi =\pi' \delta^{(h)}H\ .
$$
Therefore,
\begin{equation}\label{C_PH(x,y)}
\Gamma_{qH}(x,y)=\left( \sum_{h\geq 0}
\left(
\left[y^h\right]\Gamma_q(x,y)\right) D^{(h)}(x,y)
\right)
x \ .
\end{equation}
As far as $D^{(h)}(x,y)$ is concerned, denoting with $\gamma$ a generic Dyck prefix and with $\gamma_i$ a Dyck path, we observe that a Dyck factor $\delta^{(h)}$ can be factorized as $\delta^{(h)}=(\gamma_1 D)(\gamma_2 D)\ldots (\gamma_r D)\gamma$, with $0\leq r\leq h$, where the first $D$ step reaching height $h-i$ is highlighted, for each $i=1,2,\ldots, r$. 
From the above construction, in terms of generating functions we have:

\begin{eqnarray*}
	D^{(h)}(x,y)&=&\mathcal{DP}(x,y)y^h+ C(x) x\mathcal{DP}(x,y)y^{h-1}+\\
	&&\\
	&&C(x)^2 x^2\mathcal{DP}(x,y)y^{h-2}+
	\cdots
	+C(x)^h x^h\mathcal{DP}(x,y)y^{h-h}\\
	&&\\
	&=&\mathcal{DP}(x,y)\sum_{i=0}^hx^iy^{h-i}C(x)^i \ ,
\end{eqnarray*}
leading to
\begin{eqnarray}\label{D_F(x.y)}
D^{(h)}(x,y)
&
=
&
\frac{2}{1-2xy+\sqrt{1-4x^2}}\cdot
\frac{y^{h+1}-x^{h+1}C(x)^{h+1}}{y-xC(x)}\ .
\end{eqnarray}
Plugging (\ref{D_F(x.y)}) into (\ref{C_PH(x,y)}) we obtain:

\begin{equation*}
\Gamma_{qH}(x,y)=
\frac{2x}{\left(1-2xy+\sqrt{1-4x^2}\right)\left(y-xC(x)\right)}\sum_{h\geq 0}
\left( \left[
y^h
\right]
\Gamma_q(x,y)\right)
\left(
y^{h+1}-x^{h+1}C(x)^{h+1}(x)
\right)
\end{equation*}
which 
boils down to (\ref{C_qH}).

\bigskip

Finally, let  $\pi \in \minco{qD}$ and let $\pi'$ be the smallest prefix of $\pi$ containing $q$.~Then
$$
\pi =\pi' \alpha D
$$
where $\alpha$ is, as before, a path starting at the origin, using only $U$ and $H$ steps,
and with the additional restriction that, if $\pi'$ ends at height $h=0$, then
$\alpha \neq H^r$, $r\geq 0$ (otherwise $\pi$ would not be a Motzkin prefix, since it would terminate below the $x$-axis). In terms of generating functions, recalling the expression (\ref{A(x,y)}) for the bivariate generating function of the paths $\alpha$, we then have:

\begin{eqnarray*}
	\Gamma_{qD}(x,y)&=&\left(\sum_{h\geq 0}\left( \left[y^h\right]\Gamma_q(x,y)\right) y^hA(x,y)\right)xy^{-1}
	-
	\left( \left[y^0\right]\Gamma_q(x,y)\right)
	\sum_{i\geq 0}x^{i}xy^{-1}
	\\
	&&
	\\
	&=&\frac{x}{y}A(x,y)\Gamma_q (x,y)-\frac{x}{y}\Gamma _q(x,0)\frac{1}{1-x}\\
	&&\\
	&=&\frac{x}{y}\left(
	\frac{1}{1-x-xy}\Gamma_q (x,y)-\frac{1}{1-x}\Gamma_q (x,0)\
	\right) ,
\end{eqnarray*}
which is equal to (\ref{C_qD}).
\end{proof}

Let $\Delta_q(x)$ be the generating function of Motzkin paths avoiding a Motzkin prefix $q$ with respect to the length.
The following result gives a recursive procedure to compute $\Delta_q (x)$
.

\begin{proposition}\label{genfun}
	For any Motzkin prefix $q$, the generating function $\Delta_q(x)$ is given by:
	
	\begin{align}
	\Delta_{\epsilon}(x)&=0\\
	\label{D_qD}
	\Delta_{qD}(x)&=\Delta_q(x)+\Gamma_q (x,0)\frac{1}{1-x}\\
	\label{D_qH}
	\Delta_{qH}(x)&=\Delta_q(x)+C(x)\cdot \Gamma_q \left(x,xC(x)\right)\\
	\label{D_qU}
	\Delta_{qU}(x)&=\Delta_q(x)+\frac{1}{1-x}\Gamma_q \left(x,\frac{x}{1-x}\right)\ .
	\end{align}	
\end{proposition}

\begin{proof}
Every Motzkin path contains the empty path $\epsilon$, hence $\Delta_{\epsilon}(x)=0$.~Let $\pi \in \av{qD}$.~There are two cases: either $\pi$ avoids $q$, and such paths $\pi$ are counted by $\Delta_q(x)$, or $\pi$ contains $q$ but avoids $qD$.~In the latter case, let $\pi'$ be the smallest prefix of $\pi$ containing $q$.
Obviously $\pi'$ cannot be followed by any $D$ step in any position, otherwise the path $\pi$ would contain $qD$.~Hence the only possibility is that $\pi'$ has final height equal to $0$ and is followed by a certain number of consecutive $H$ steps. In other words, $\pi$ can be factorized as
$$
\pi =\pi'H^i\ ,
$$
with $i\geq 0$. In terms of generating functions, the above argument leads to:

$$
\Delta_{qD}(x)=\Delta_q(x)+\left( \left[y^0\right]\Gamma_q (x,y)\right) \sum_{i\geq 0}x^{i}
=
\Delta_q(x)+\Gamma_q (x,0)\frac{1}{1-x}
$$
which is equation (\ref{D_qD}).

\bigskip
Suppose that $\pi \in \av{qH}$.~If $\pi$ also avoids $q$, then, as in the previous case, we obtain the generating function $\Delta_q(x)$.~Otherwise, $\pi$ can be decomposed as its smallest prefix $\pi'$ containing $q$, ending at height $h\geq 0$, followed by a path starting from height $h$, using only $U$ and $D$ steps and ending on the $x$-axis. This path is easily seen to be the reverse of a Dyck prefix having final height $h$, hence:
\begin{equation}\label{pippo}
\Delta_{qH}(x)=\Delta_q(x)+
\sum_{h\geq 0}\left(\left[y^h\right]\Gamma_q (x,y)\right)
\left( \left[y^h\right]\mathcal{DP}(x,y)\right)\ .
\end{equation}
Since we have
$$
\left[y^h\right]\mathcal{DP}(x,y)=
\frac{2}{1+\sqrt{1-4x^2}}\left(\frac{2x}{1+\sqrt{1-4x^2}}\right)^h\ ,$$
plugging the above expression into (\ref{pippo}), and observing that $\frac{2}{1+\sqrt{1-4x^2}}=C(x)$, we get:
$$
\Delta_{qH}(x)=\Delta_q(x)+C(x) \sum_{h\geq 0}\left( \left[y^h\right]\Gamma_q (x,y)\right)
\left(
xC(x)
\right)^h
$$
which is equivalent to (\ref{D_qH}).

\bigskip
Finally, let $\pi \in \av{qU}$.~If $\pi$ contains $q$, as usual let $\pi'$ be the smallest prefix of $\pi$ containing $q$.~ The path $\pi$ can be written as $\pi'$, which ends at height $h\geq 0$, followed by a path starting from height $h$ and using only $D$ and $H$ steps. This latter path is the reverse of a path $\alpha$ starting at the origin, using only $U$ steps and $H$ steps and ending at height $h$.~Recalling once more the expression (\ref{A(x,y)}) of the bivariate generating function $A(x,y)$ of such paths, we obtain:
$$
\Delta_{qU}(x)=\Delta_q(x)+\sum_{h\geq 0}
\left(
\left[
y^h
\right]
\Gamma_q (x,y)\right)
\left( \left[
y^h
\right]
A(x,y)
\right)\ .
$$
Since $\left[
y^h
\right]
A(x,y)
=\frac{1}{1-x}
\left(
\frac{x}{1-x}
\right)^h
$
, we get:
$$
\Delta_{qU}(x)=\Delta_q(x)+
\frac{1}{1-x}
\sum_{h\geq 0}
\left(
\left[
y^h
\right]
\Gamma_q (x,y)\right)
\left(
\frac{x}{1-x}
\right)^h
\ ,
$$
which is (\ref{D_qU}).
\end{proof}

As a consequence of Propositions \ref{Prop_C_P} and \ref{genfun}, we get that,
for a given pattern $q$, 
the generating function $\Delta_q (x)$ of Motzkin paths avoiding $q$ is rational over $x$ and $C(x)$,
which is the statement of Theorem \ref{thm:motzkinarealmostrational}.

\section{Conclusion}
\label{sec:conclusion}

The main results of the present paper, 
namely an algorithm to determine the generating function of Motzkin paths avoiding a set of patterns, 
and the proof that such generating function is rational over $x$ and $C(x)$ 
(at least in the case of a single pattern),
may be seen as a further step towards a deeper investigation of pattern avoidance in lattice paths. 
For instance, 
the same approach developed here can be pursued for Schr\"oder paths, for which pattern avoidance has been first studied in Cioni and Ferrari~\cite{cioniferrari:schroder}, 
thus getting completely analogous results 
(in particular, the same technique described in Section \ref{sec:proof} can be exploited to show that 
the generating function of Schr\"oder paths avoiding a single pattern is also rational over $x$ and $C(x)$).
It would then be interesting to find analogous results in the case of an arbitrary set of steps.

Another issue that seems worth investigation is the asymptotic behavior of classes of pattern-avoiding Motzkin paths. 
In the case of Dyck paths, 
in Bacher et al.~\cite{bacheretal:dyck} it is shown that, 
regardless of the specific pattern to be avoided, the asymptotic behavior of all classes of Dyck paths avoiding a single pattern is the same (and it is polynomial).
Having a similar result for pattern-avoiding Motzkin paths would be desirable.

There are some papers, such as Asinowki et al.~\cite{asinowskietal:kernel} and Asinowki et al.~\cite{asinowskietal:pattern}, which delevop a methodology based on automata and a variant of the kernel method to study lattice paths avoiding a \emph{consecutive} pattern (that is a pattern whose element are adjacent in the path). It seems conceivable that a similar approach could be fruitful also in the case of generic patterns.  

\bigskip

%
%
%
%
%

\bibliography{mybibfile}

\begin{thebibliography}{10}
\expandafter\ifx\csname url\endcsname\relax
  \def\url#1{\texttt{#1}}\fi
\expandafter\ifx\csname urlprefix\endcsname\relax\def\urlprefix{URL }\fi
\expandafter\ifx\csname href\endcsname\relax
  \def\href#1#2{#2} \def\path#1{#1}\fi

\bibitem{donagheyshapiro:motzkin}
R.~Donaghey, L.~W. Shapiro, Motzkin numbers, J. Combinatorial Theory Ser. A
  23~(3) (1977) 291--301.

\bibitem{asinowskietal:kernel}
A.~Asinowski, A.~Bacher, C.~Banderier, B.~Gittenberger, Analytic combinatorics
  of lattice paths with forbidden patterns, the vectorial kernel method, and
  generating functions for pushdown automata, Algorithmica 82 (2020) 386--428.

\bibitem{sapounakisetal:string}
A.~Sapounakis, I.~Tasoulas, P.~Tsikouras, Counting strings in {D}yck paths,
  Discrete Math. 307 (2007) 2909--2924.

\bibitem{bacheretal:dyck}
A.~Bacher, A.~Bernini, L.~Ferrari, B.~Gunby, R.~Pinzani, J.~West, The {D}yck
  pattern poset, Discrete Math. 321 (2014) 12--23.

\bibitem{flajoletsedgewick:ac}
P.~Flajolet, R.~Sedgewick, Analytic combinatorics, Cambridge University Press,
  Cambridge, 2009.

\bibitem{sloane}
N.~J.~A. Sloane, \href{oeis.org}{The on-line encyclopedia of integer
  sequences}.
\newline\urlprefix\url{oeis.org}

\bibitem{bean:thesis}
C.~Bean, {Finding structure in permutation sets}, Ph.D. thesis, Reykjavik
  University (June 2018).

\bibitem{bean:cssrepo}
C.~Bean, J.~S. Eliasson, E.~Nadeau, J.~Pantone, H.~Ulfarsson,
  \href{https://doi.org/10.5281/zenodo.4946832}{Permutatriangle/comb\_spec\_searcher:
  Version 4.0.0} (June 2021).
\newline\urlprefix\url{https://doi.org/10.5281/zenodo.4946832}

\bibitem{bean:motzkinrepo}
C.~Bean, \href{https://doi.org/10.5281/zenodo.5159935}{Permutatriangle/motzkin:
  Version 1.0.0} (August 2021).
\newline\urlprefix\url{https://doi.org/10.5281/zenodo.5159935}

\bibitem{cioniferrari:schroder}
L.~Cioni, L.~Ferrari, Enumerative results on the {S}chr\"oder pattern poset,
  in: A.~Dennunzio, E.~Formenti, L.~Manzoni, A.~Porreca (Eds.), Cellular
  Automata and Discrete Complex Systems. AUTOMATA 2017, Vol. 10248 of Lecture
  Notes in Computer Science, Springer, Cham, 2017, pp. 56--67.

\bibitem{asinowskietal:pattern}
A.~Asinowski, C.~Banderier, V.~Roitner, Generating functions for lattice paths
  with several forbidden patterns, S\'em. Lothar. Combin. 84B (2020) Article
  \#95, 12 pp.

\end{thebibliography}

\end{document}